\documentclass{amsart}

\title{The Smith-Toda complex $V(4)$ exists for $p \geq 11$}
\author{Vigleik Angeltveit}
\address{Mathematical Sciences Institute \\
Australian National University \\
Canberra, ACT 0200 \\
Australia}

\usepackage{amsxtra}
\usepackage{amsfonts}
\usepackage[latin1]{inputenc}
\usepackage{graphicx}
\usepackage{amsmath,amssymb,latexsym,amsthm,mathrsfs}
\usepackage[all]{xy}
\usepackage{verbatim}
\usepackage{hyperref}

\newtheorem{theorem}{Theorem}[section]
\newtheorem{thm}[theorem]{Theorem}
\newtheorem{lemma}[theorem]{Lemma}

\newtheorem{prop}[theorem]{Proposition}

\theoremstyle{definition}

\newtheorem{remark}[theorem]{Remark}
\newtheorem{example}[theorem]{Example}
\newtheorem{notation}[theorem]{Notation}

\numberwithin{equation}{section}

     \newcommand{\bF}{\mathbb{F}}          
         \newcommand{\bZ}{\mathbb{Z}}

\newcommand{\sma}{\wedge} 

\newcommand{\xto}{\xrightarrow}

\newcommand{\wt}{\widetilde}

\newcommand{\Ext}{\textnormal{Ext}}

\begin{document}

\begin{abstract}
We show that the Smith-Toda complex $V(4)$ exists for all primes $p \geq 11$. This improves on the recently announced result by OpenAI that $V(4)$ exists at the prime $1009$.
\end{abstract}

\maketitle

\section{Introduction}
OpenAI recently released a massive trove of mathematics papers. Two of those papers, \cite{OAI:V4} and \cite{OAI:Vn}, are about the existence of Smith-Toda complexes. The first one \cite{OAI:V4} proves that $V(4)$ exists at $p = 1009$, and the second one proves that for any $n$ there exists a prime $p$ for which $V(n)$ exists.

We thought both of these papers looked a bit suspicious. We have not found any mathematical errors, but both of them make some strange choices in what they prove and how they do it. The first result really should say that $V(4)$ exists for all $p \geq p_0$, and in any case $1009$ is suspiciously high. The second result should say that $V(n)$ exists for all sufficiently large $p$. Moreover, the second paper uses \emph{ultrafilters} throughout. We have not been able to find a good reason for this. We speculate that ultrafilters came from one of the references, and the AI model left them in even though the reason that particular reference uses ultrafilters is irrelevant to the problem at hand.

Here we choose to focus on the first of the two OpenAI papers, and we make the following improvement:

\begin{thm} \label{t:main}
For any prime $p \geq 7$ the element $v_4$ in bidegree $(0, 2p^4-2)$ of the Adams-Novikov spectral sequence converging to $\pi_* V(3)$ is a permanent cycle and represents an element $v_4 \in \pi_{2p^4-2} V(3)$.

For any prime $p \geq 11$ there exists a finite $p$-local spectrum $X$ and an isomorphism
\[
 BP_* X \cong BP_*/(p, v_1, v_2, v_3, v_4)
\]
of graded $BP_* BP$-comodules.
\end{thm}

The OpenAI paper has the identical conclusion to the second part of Theorem \ref{t:main} for $p = 1009$. In addition to improving their result to cover more primes, we significantly simplify their argument in a few places. Their Section 4 has been replaced by a reference to Toda \cite{To71}. Their Section 6 has been replaced by a much simpler May spectral sequence calculation. And their Section 7 has been replaced by a simple Toda bracket juggling trick.

We refer to a finite $p$-local spectrum $X$ with $BP_* X \cong BP_*/(p, v_1,\ldots, v_n)$ as a Smith-Toda complex and denote it by $V(n)$. By a ring spectrum we mean a spectrum with a unital multiplication; no associativity or commutativity is implied. From \cite[Theorems 1.1 and 4.4]{To71} we have the following: The Smith-Toda complex $V(0)$ is the mod $p$ Moore spectrum, which exists for all $p$ and is a ring spectrum for $p \geq 3$. Next, $V(1)$ exists for $p \geq 3$ and is a ring spectrum for $p \geq 5$. It is always a module spectrum over $V(0)$, where we interpret module spectrum in a similar way as a ring spectrum. The spectrum $V(2)$ exists for $p \geq 5$ and is a ring spectrum for $p \geq 7$. It is always a module spectrum over $V(1)$. Finally, $V(3)$ exists for $p \geq 7$ and is a ring spectrum for $p \geq 11$. It is always a module spectrum over $V(2)$.

\subsection{Proof strategy}
The proof strategy is as follows: In the $E_2$-term of the Adams-Novikov spectral sequence (ANSS) converging to $\pi_* V(3)$, there is an element which, if it survives, represents $v_4$. We prove that $v_4$ does in fact survive whenever $V(3)$ exists. Then we use the ring structure on $V(3)$, which only exists for $p \geq 11$, to construct a self map $v_4 : \Sigma^{2p^4-2} V(3) \to V(3)$ whose cofiber is $V(4)$ as the composite
\begin{equation} \label{eq:upgradetoselfmap}
 \Sigma^{2p^4-2} V(3) \xto{\cong} S^{2p^4-2} \sma V(3) \xto{v_4 \sma 1} V(3) \sma V(3) \to V(3).
\end{equation}

So far this is identical to the strategy in the first OpenAI Smith-Toda paper, except that we rely on Toda for the fact that $V(3)$ exists and has a ring structure rather than proving it ourselves.

We emphasise that the condition that $p \geq 11$ comes from needing a ring structure on $V(3)$. This is similar to the various classical cases: We always have $v_1 \in \pi_{2p-2} V(0)$, but we need $p \geq 3$ to upgrade this to a $v_1$-self map on $V(0)$. We have $v_2 \in \pi_{2p^2-2} V(1)$ whenever $V(1)$ exists but we need $p \geq 5$ to upgrade it to a $v_2$-self map on $V(1)$. And we have $v_3 \in \pi_{2p^3-2} V(2)$ whenever $V(2)$ exists but we need $p \geq 7$ to upgrade it to a $v_3$-self map on $V(2)$.

The May spectral sequence (MSS) is a tri-graded spectral sequence which converges to the $E_2$-term of the ANSS through a range of degrees. To show that the class representing $v_4$ survives the ANSS, we check that there are exactly $9$ classes in the $E_1$-term of the MSS which contribute to the right bidegree in the ANSS to make them potential targets of a differential on $v_4$. Out of these, exactly $7$ die in the MSS. Our argument for this is much more elementary than the one given in the OpenAI paper. This leaves exactly $2$ possible targets of a differential on $v_4$ in the ANSS, of length $(2p-1)$ and $p(2p-2)+1$.

The argument for why the first of these differentials vanishes (Theorem \ref{t:h14injective}) relies on a relation supplied by $v_5$, and is also given in the OpenAI paper. To show that the second one vanishes, we show that the potential target dies at an earlier page of the ANSS. This was the most complicated part of the OpenAI paper; we give a much shorter argument (Theorem \ref{t:nolongdiff}).

\subsection{AI assistance}
ChatGPT (OpenAI) was used extensively as an interactive research assistant during the development of this work, including for checking calculations, locating relevant literature, and exploring simplifications of the arguments in the OpenAI Smith-Toda papers. In particular, conversations with ChatGPT helped lead to the reformulation of the final Adams-Novikov obstruction in terms of Toda's classical bracket-juggling argument. The author independently checked the mathematical arguments and is solely responsible for the contents of the paper.

\section{Background}
We recall some of the conventions and results from \cite[Sections 2-3]{OAI:V4} and make some small improvements. With the exception of the zero digit modification to the May filtration this material is all standard and can be found, for example, in Ravenel's book \cite{Ra04}.

\begin{lemma}
Suppose $V(n-1)$ exists. The map
\[
 \sigma : BP_* V(n-1) = \bF_p[v_n, v_{n+1},\ldots] \to \bF_p
\]
induces a map from the $E_2$-term of the ANSS
\[
 E_2^{*,*} = \Ext_{BP_* BP}^{*,*}(BP_*, BP_* V(n-1))
\]
to
\[
 \Ext_{BP_* BP}^{*,*}(BP_*, \bF_p) \cong H^*(P)
\]
which is an isomorphism in internal degree less than $|v_n|$.

Here $P = \bF_p[t_1, t_2, \ldots]$ with coproduct $\Delta(t_i) = \sum_{\ell=0}^i t_\ell \otimes t_{i-\ell}^{p^\ell}$.
\end{lemma}

Because $BP_* V(n-1)$ is concentrated in degrees divisible by $q = 2p - 2$, we can grade $E_2^{*,*}$ by cohomological degree and weight, with $v_i$ and $t_i$ in weight $u_i = \frac{|v_i|}{q} = \frac{p^i-1}{p-1}$ and weight $w$ corresponding to internal degree $qw$. Then $E_2^{s,w}$ contributes to $\pi_{qw-s} V(n-1)$. It follows that $d_r = 0$ unless $r = kq+1$, in which case $d_{kq+1} : E_{kq+1}^{s,w} \to E_{kq+1}^{s+kq+1, w+k}$.

By $H^*(P)$ we mean the cohomology of the standard cobar complex. This is bigraded by cohomological and internal degree, but we suppress the second grading from the notation.

The isomorphism breaks down in degree $|v_n|=2p^n-2$, where we have a class $v_n \in E_2^{0, u_n}$. Since $v_n$ is in filtration $0$ it cannot be a boundary, and the only way for it to die is if it supports a non-zero differential.

\begin{lemma} \label{l:v5rel}
Suppose $V(n-1)$ exists. In the $E_2$-term of the Adams-Novikov spectral sequence converging to $\pi_* V(n-1)$ we have $E_2^{*,*} \cong H^*(P)[v_n]$ in weight less than $u_{n+1}$, and in weight exactly $u_{n+1}$ there is a relation
\[
 h_{1,n} v_n = h_{1,0} v_n^p.
\]
\end{lemma}

\begin{proof}
The first part follows because when we work mod $(p,v_1,\ldots,v_{n-1})$ we have $\eta_R(v_n) = v_n$, and more generally $\eta_R(v_n^a) = v_n^a$. So below the weight of $v_{n+1}$ the polynomial generator $v_n$ just comes along.

The relation follows from the right unit on $v_{n+1}$: By \cite[Proof of Theorem 4.3.2]{Ra04}, if we work mod $(p, v_1, \ldots, v_{n-1})$ we have $\eta_R(v_{n+1}) = v_{n+1} + v_n t_1^{p^n} - v_n^p t_1$. Hence $v_{n+1} - \eta_R(v_{n+1})$ gives the relation.
\end{proof}

\begin{remark}
In a given cohomological degree $s$ we can push this a little bit further, and $E_2^{s,*} \cong H^s(P)[v_n]$ in weight less than $u_{n+1}+s-2$.
\end{remark}

\begin{remark}
We also need $V(n-1)$ to be a ring spectrum in order to interpret this topologically and to take advantage of the Leibniz rule to compute differentials. It is also possible to rely on a $V(n-2)$-module structure on $V(n-1)$ for some arguments.
\end{remark}

\begin{lemma}[{\cite{Ma66}, \cite[Theorem 3.2.5]{Ra04}}] \label{l:MaySSE1term}
There is a May spectral sequence (MSS) converging to $H^*(P)$ with $E_1$-term
\[
 E_1^{*,*,*} = \Lambda(h_{i,j} \, : \, i \geq 1, j \geq 0) \otimes \bF_p[b_{i,j} \, : \, i \geq 1, j \geq 0],
\]
obtained by putting $t_i^{p^j}$ in filtration $2i-1$.
\end{lemma}

Just as we did for the ANSS we can replace one of the gradings by the weight grading. To distinguish differentials in the MSS from differentials in the ANSS we denote the $r$'th differential in the MSS by $\delta_r$ and the $r$'th differential in the ANSS by $d_r$.

The OpenAI paper also uses a modified filtration \cite[Lemma 3.1]{OAI:V4} to show that one particular MSS differential is trivial. To describe it, let us write $\nu_0$ for the zero'th $p$-adic digit of an exponent, counted with multiplicity. For example, $\nu_0(t_1^{p+2} | t_2^3) = 2 + 3 = 5$ for $p > 3$. The cobar differential does not increase $\nu_0$, and sometimes decreases it, so it induces a filtration on the cobar complex. To refine equal $\nu_0$-weight we use the standard May filtration where $t_i^{p^j}$ has filtration $2i-1$. Equivalently, for sufficiently large $M$ we use the filtration
\[
 F_M = F_{May} + M \nu_0.
\]
The reduced coproduct still reduces the filtration by exactly $1$, so this filtration produces the same associated graded and the same $\delta_1$-differential.

\begin{notation} \label{n:xacybzABC}
We follow \cite[Section 3]{OAI:V4} and define
\begin{eqnarray*}
 x = h_{1,0} \quad a = h_{2,0} \quad c = h_{3,0} \quad y = h_{1,1} \quad b = h_{2,1} \quad z = h_{1,2} \\
 A = b_{1,0} \quad B = b_{2,0} \quad C = b_{1,1}
\end{eqnarray*}
The order seems idiosyncratic, but it makes sense if we also define $h_{[a,b)} = h_{b-a,a}$. With the open intervals notation, $\delta_1(h_{[a,b)}) = \sum\limits_{a < \ell < b} h_{[a,\ell)} h_{[\ell,b)}$. In particular
\[
 \delta_1(a) = xy, \quad \delta_1(b) = yz, \quad \delta_1(c) = xb + az.
\]
\end{notation}

\begin{example}
In the standard MSS we have $\delta_{2p-1}(B) = zA - yC$. If we use the modified filtration with $M > 0$ we instead have $\delta_{2p-1}(B) = zA$.
\end{example}

\begin{remark}
The OpenAI paper chooses $M = 3p$, but there is nothing special about $3p$. Any sufficiently large $M$ will do.
\end{remark}

\section{Some May spectral sequence calculations}
\begin{prop}[{\cite[Lemma 5.2, Proposition 5.5]{OAI:V4}}] \label{p:v4targets}
The only possible targets of a differential on $v_4$ on the $E_1$-page of the May spectral sequence converging to $H^{*,*}(P)$ are the $9$ classes listed in Table \ref{t:v4-targets}.
\end{prop}

\begin{table}[ht]
\centering
\begin{tabular}{c|l|l|l}
$k$ & May $E_1$-class & Method & Argument \\ \hline

$1$
& $acb\,B^{p-2}$
& MSS
& supports $\delta_1$ \\

$1$
& $acybz\,B^{p-3}$
& ANSS
& Theorem \ref{t:h14injective} \\ \hline

$2$
& $xac\,BA^{p-2}C^{p-2}$
& MSS
& supports $\delta_{2p-1}$ \\

$2$
& $xacyz\,A^{p-2}C^{p-2}$
& MSS
& $\delta_1$-boundary \\ \hline

$p-1$
& $yB^3A^{N_0-3}$
& MSS
& supports $\delta_{4p-3}$ \\

$p-1$
& $bB^2A^{N_0-2}$
& MSS
& supports $\delta_1$ \\

$p-1$
& $ybz\,BA^{N_0-2}$
& MSS
& $\delta_{2p-1}$-boundary \\ \hline

$p$
& $cyb\,A^{N+1}$
& MSS
& supports $\delta_1$ \\

$p$
& $ayb\,BA^N$
& ANSS
& Theorem \ref{t:nolongdiff}
\end{tabular}
\caption{Elimination of the possible targets of a differential on $v_4$.
Here $N_0=(p-1)^2$ and $N=p(p-1)-2$.}
\label{t:v4-targets}
\end{table}

Compare the following proof to the proof of \cite[Lemma 5.2 and Proposition 5.5]{OAI:V4}. We cannot simply copy their proof, as some of the inequalities they rely on require $p$ to be large.
\begin{proof}

We take advantage of the weight grading, with $v_4 \in E_2^{0, p^3+p^2+p+1}$. To start with, $d_r$ can only be nonzero when $r = kq + 1$ for some $k \geq 1$, and then $d_r(v_4)$ has weight grading $p^3+p^2+p+1+k$ and cohomological degree $r$. Now it is a simple matter of checking which monomials in the $E_1$-term of the MSS are in the right bidegree.

In addition to the variables $x$, $a$, $c$, $y$, $b$, $z$, $A$, $B$, $C$ from Notation \ref{n:xacybzABC}, we have variables $h_{4,0}$, $h_{3,1}$, $h_{2,2}$, $h_{1,3}$, $b_{3,0}$, $b_{2,1}$ and $b_{1,2}$ which could contribute to weight $p^3+p^2+p+1+k$. We write the weight in base $p$, so the target weight is $(1, 1, 1, 1+k)$ and we keep in mind that base $p$ addition might involve a carry. Each variable has weight which, in base $p$, consists of some interval of $1$'s surrounded by $0$'s:
\begin{eqnarray*}
 w(a) & = & (0, 0, 1, 1) \qquad
 w(b)   =   (0, 1, 1, 0) \qquad
 w(c)   =   (0, 1, 1, 1) \\
 w(x) & = & (0, 0, 0, 1) \qquad
 w(y)   =   (0, 0, 1, 0) \qquad
 w(z)   =   (0, 1, 0, 0) \\
 w(A) & = & (0, 0, 1, 0) \qquad
 w(B)   =   (0, 1, 1, 0) \qquad
 w(C)   =   (0, 1, 0, 0) \\
 w(h_{4,0}) & = & (1, 1, 1, 1) \qquad
 w(h_{3,1})   =   (1, 1, 1, 0) \qquad
 w(h_{2,2})   =   (1, 1, 0, 0) \\
 w(h_{1,3}) & = & (1, 0, 0, 0) \qquad
 w(b_{3,0})   =   (1, 1, 1, 0) \qquad
 w(b_{2,1})   =   (1, 1, 0, 0) \\
 w(b_{1,2}) & = & (1, 0, 0, 0)
\end{eqnarray*}

We first exclude the possibility that any monomial could have a factor of $h_{4,0}$, $h_{3,1}$, $h_{2,2}$, $h_{1,3}$, $b_{3,0}$, $b_{2,1}$ or $b_{1,2}$. These variables all have weight at least $p^3 = (1, 0, 0, 0)$, so if we have one of these we cannot have any carry to the $p^3$ position and we can have at most $1$ variable with weight that has a $1$ in the $p^2$ position. If we have such a variable with a $1$ in the $p^2$ position then there can be no carry from the $p$ position, and this quickly excludes any possible solution. If there are no such variables with a $1$ in the $p^2$ position we must have $p+1$ variables with a $1$ in the $p$ position, and there are not enough variables available.

The enumeration now follows by a case by case analysis, treating one value of $k$ at a time. If we have $e$ exterior generators then we must have $k(p-1) - \frac{e-1}{2}$ polynomial generators, and we must have $e \in \{1, 3, 5\}$. The variables $a$, $c$ and $x$ are special in that they are the only ones that contribute to the $p^0$ position of the weight.

If $k = 1$, we must have $1 + k = 2$ in the $p^0$ position so we must have exactly $2$ of $a$, $c$ and $x$. If we have $3$ exterior generators we must have $p-2$ polynomial generators, and if we have $5$ exterior generators we must have $p-3$ polynomial generators. The polynomial generators must all be $B$ to get a high enough weight, so we quickly find the $2$ monomials on our list and no more.

If $k = 2$, we must have $1 + k = 3$ in the $p^0$ position so we must have all $3$ of $a$, $c$ and $x$. We find the $2$ monomials on our list and no more by a similar analysis.

If $3 \leq k \leq p-2$ it is impossible to get a $1 + k$ in the $p^0$ position.

If $k = p - 1$, we cannot have any of $a$, $c$ and $x$ and we must have either $1$ or all $3$ of $b$, $y$ and $z$. Most of the polynomial generators must be $A$, otherwise the weight becomes too big, and we find that the $3$ monomials on our list are the only possible ones.

If $k = p$, we must have $1$ of $a$, $c$ and $x$. Once again most of the polynomial generators must be $A$, and again we find the $2$ monomials on our list and no more.

If $k > p$ a similar analysis shows that we either get the wrong digit in the $p$ position or the weight becomes too large, so there are no more solutions.
\end{proof}

\begin{thm} \label{t:MaySSdiff}
We have the following differentials in the May spectral sequence converging to $H^*(P)$:
\begin{eqnarray*}
 \delta_1(acbB^{p-2}) & = & xycbB^{p-2} + acyzB^{p-2} \\
 \delta_1(xacbA^{p-2}C^{p-2}) & = & xacyzA^{p-2}C^{p-2} \\
 \delta_1(bB^2A^{N_0-2}) & = & yzB^2A^{N_0-2} \\
 \delta_1(cybA^{N+1}) & = & -azybA^{N+1} \\
 \delta_{2p-1}(xacBA^{p-2}C^{p-2}) & = & - xaczA^{p-1}C^{p-2} + xacyA^{p-2}C^{p-1} \\
 \delta_{2p-1}(ybB^2A^{N_0-3}) & = & 2ybzBA^{N_0-2} \\
 \delta_{4p-3}(yB^3A^{N_0-3}) & = & -6bz BA^{N_0-1} + 6byBA^{N_0-2}C.
\end{eqnarray*}

\end{thm}

\begin{proof}
The $\delta_1$-differentials are immediate, and the $\delta_{2p-1}$-differentials follow from the usual differential $\delta_{2p-1}(B) = zA - yC$.

The differential of length $4p-3$ comes from a zig-zag: Let $w = yB^3 A^{N_0-3}$ and let $w' = -3bB^2 A^{N_0-2}$. Then we have
\[ \xymatrix{
 w \ar[r]^-{\delta_{2p-1}} & -3yzB^2A^{N_0-2} \\
 w' \ar[ur]^-{\delta_1} \ar[r]_-{\delta_{2p-1}} & 6bzBA^{N_0-1} - 6byBA^{N_0-2}C
}\]
This implies that $\wt{w} = w - w'$ supports the claimed $\delta_{4p-3}$. There is no possible target in a filtration between these, so the zig-zag really does calculate $\delta_{4p-3}$. Compare to \cite[Proof of Proposition 6.6]{OAI:V4}, and note that they also end up with a factor of $6$.
\end{proof}

This takes care of all the classes in Table \ref{t:v4-targets} that are listed as dying in the MSS.

\begin{thm} \label{t:h14injective}
Suppose $V(3)$ exists. In the ANSS converging to $\pi_* V(3)$ we have $d_{q+1}(v_4) = 0$.
\end{thm}

\begin{proof}
Let $u = acybz B^{p-3}$. If $d_{q+1}(v_4) \neq 0$ then the only possibility is $d_{q+1}(v_4) \doteq u$, where $\doteq$ denotes equality up to a unit in $\bF_p$. By Lemma \ref{l:v5rel} we have $h_{1,4} v_4 = h_{1,0} v_4^p$ on the $E_2$-term of the ANSS and hence also on the $E_{q+1}$-term. This is in weight $u_5$.

Now we compute $d_{q+1}$ of both sides, and by using that $V(3)$ is a $V(2)$-module spectrum we get
\begin{equation} \label{eq:dofv5rel}
 d_{2p-1}(h_{1,4}) v_4 - h_{1,4} d_{2p-1}(v_4) = d_{2p-1}(h_{1,0})v_4^p - h_{1,0} d_{2p-1}(v_4^p).
\end{equation}
We can apply $\sigma$ to this to obtain
\begin{equation} \label{eq:sigmadofv5rel}
 h_{1,4} \sigma(d_{2p-1}(v_4)) = h_{1,0} \sigma(d_{2p-1}(v_4^p)).
\end{equation}

This equation is in cohomological degree $2p$ and weight $u_5+1$. If $p \geq 11$, so $V(3)$ is a ring spectrum, we automatically have $d_{2p-1}(v_4^p) = 0$. If $p = 7$ then we cannot interpret $v_4^p$ as a product, so this does not immediately follow. Instead we make the following argument:

This class lives in $H^{2p-1, u_5}(P)$, so it suffices to compute $H^*(P)$ in that bidegree. By a similar argument as in the proof of Proposition \ref{p:v4targets}, the May $E_1$-page in that bidegree is spanned by the following $3$ monomials:
\begin{eqnarray*}
 r_1 & = & h_{4,0} h_{1,1} h_{3,1} b_{3,0}^{p-2} \\
 r_2 & = & h_{4,0} h_{1,2} h_{2,2} b_{2,1}^{p-2} \\
 r_3 & = & h_{4,0} h_{1,1} h_{2,1} h_{3,1} h_{1,3} b_{3,0}^{p-3}.
\end{eqnarray*}
To check this, note that in base $p$ we have $u_5= (1, 1, 1, 1, 1)$. To get there, before carrying we have to have $(0, p, p, p+1, 1)$ or $(0, p, p+1, 1, 1)$. We have to have either $3$ or $5$ exterior generators, and now the case-by-case analysis is straightforward.

These all support a $\delta_1$-differential, and the $\delta_1$-images are linearly independent, so the $E_2$-term of the May spectral sequence is trivial in that bidegree and we conclude that $\sigma(d_{2p-1}(v_4^p))$ is also $0$ at $p = 7$.

Hence it suffices to show that $h_{1,4} u$ in the $E_1$-term of the MSS is not the target of a May differential, because then $h_{1,4} u$ survives to $E_{q+1}^{*,*}$ and if $d_{q+1}(v_4) \doteq u$ we get a contradiction. The possible sources of a MSS differential on $h_{1,4}u$ are given in \cite[Table 2]{OAI:V4}. The same base-$p$ digit enumeration as in Proposition \ref{p:v4targets} shows that for every $p \geq 7$ this is a complete list. For convenience we repeat the them here:
\begin{eqnarray*}
 f_3 & = & h_{[0, 2)} h_{[0,4)} h_{[1,4)} b_{3,0}^{p-2} \\
 f_5 & = & h_{[0, 3)} h_{[0, 4)} h_{[1,2)} h_{[1,4)} h_{[3,4)} b_{3,0}^{p-3} \\
 f_5' & = & h_{[0, 2)} h_{[0, 4)} h_{[1,3)} h_{[1,4)} h_{[3,4)} b_{3,0}^{p-3} \\
 f_5'' & = & h_{[0, 2)} h_{[0, 4)} h_{[1,2)} h_{[1,4)} h_{[2,4)} b_{3,0}^{p-3} \\
 f_7 & = & h_{[0, 2)} h_{[0, 4)} h_{[1,2)} h_{[1,3)} h_{[1,4)} h_{[2,4)} h_{[3,4)} b_{3,0}^{p-4} \\
\end{eqnarray*}
To check this, note that the target weight in base $p$ is $(1, 1, 1, 1, 2)$. To get there, before carrying we have to have $(0, p, p, p+1, 2)$ and we must have $3$, $5$ or $7$ exterior generators. Once again the case-by-case analysis is straightforward.

The first $4$ generators die in the MSS because they support linearly independent $\delta_1$-differentials, and this leaves the one class $f_7$ which could, a priori, support a May differential hitting $h_{1,4}u$. To see that this cannot happen, \cite{OAI:V4} uses the zero digit modified May filtration we described after Lemma \ref{l:MaySSE1term} with $M = 3p$. Any sufficiently large $M$ works. Since $\nu_0(f_7) = p^2 - 4p + 2 < \nu_0(h_{1,4} u) = p^2 - 3p + 2$, there cannot be such a differential in the modified May spectral sequence and hence either $u$ dies in the MSS or $h_{1,4} u$ survives to the ANSS.
\end{proof}

\begin{remark}
The above class $u$ survives the MSS. We do not know if it is a permanent cycle in the ANSS, as it could potentially support a $d_{q+1}$-differential.
\end{remark}

\begin{remark}
The above argument makes the following precise: The leading term of $f_7$ has fewer zero-digit factors than the leading term of $h_{1,4} u$, and because the cobar complex coproduct cannot increase the number of zero-digit factors, it cannot map $f_7$ to $h_{1,4} u$. It is possible, but more awkward, to make this argument precise without the language of spectral sequences.
\end{remark}


\begin{lemma} \label{l:smallstems}
We have $\pi_{q(2p^2+4p+1)-5} V(3) \cong \bZ/p\{\gamma\}$, where $\gamma$ is detected by the infinite cycle $\theta = aybB \in E_2^{5, 2p^2+4p+1}$. Moreover, $\pi_{q(2p^2+4p+2)-6} V(3) = 0$.
\end{lemma}

In addition to being prime-independent this is a little bit stronger than the statement in the OpenAI paper \cite[Lemma 7.1]{OAI:V4}. The point of the second part is that $\alpha_1 \gamma \in \pi_{q(2p^2+4p+2)-6} V(3)$, so this guarantees that $\alpha_1 \gamma = 0$.

\begin{proof}
The $E_1$-term of the MSS converging to bidegree $(5, 2p^2+4p+1)$ has exactly two monomials $\theta = aybB$ and $cybA$. The second of these supports a $\delta_1$, and there is nothing in bidgree $(4, 2p^2+4p+1)$ so $aybB$ is not a MSS boundary. To check that $\theta$ survives the MSS, we look in bidegree $(6, 2p^2+4p+1)$. On the $E_1$-term of the MSS there are $5$ monomials, but only $ayB^2$ survives to $E_2$. This is in higher May filtration than $\theta$, so $\theta$ cannot support a May differential hitting it. Hence $E_2^{5,2p^2+4p+1} \cong \bZ/p\{\theta\}$.

Next we prove that $\theta$ is an infinite cycle. Since every differential has length $kq+1$ for some $k \geq 1$, our class $\theta$ cannot be a boundary. So it suffices to show that $E_2^{5+kq+1, 2p^2+4p+1+k} = 0$, and for that it suffices to show that there are no monomials in that bidegree in the $E_1$-term of the MSS. The base $p$ representation of the target weight is $(2, 4, 1+k)$.

For $k = 1$ the target weight is $(2, 4, 2)$. So we need exactly two of the exterior generators $a$, $c$ and $x$, and a total of either $2$ or $4$ exterior generators and either $p+1$ or $p$ polynomial generators. When adding up the total weight we have to do one carry from the $p$ digit to the $p^2$ digit, so without carry the target weight is $(1, p+4, 2)$. In either case we cannot get to $p+4$ in the $p$ digit, so the target is trivial.

For $k = 2$ the target weight is $(2, 4, 3)$ so we need all three of the exterior generators $a$, $c$ and $x$. This already gives a $1$ in the $p^2$ digit, and since every remaining generator contributes a weight of at least $p$ the weight becomes too large. Similarly, for $k > 2$ the weight becomes too large. This shows that $\theta$ survives the ANSS.

To show that $\pi_{q(2p^2+4p+1)-5} V(3)$ is exactly $\bZ/p$ it is enough to show that there is nothing else in the ANSS contributing to that degree, or that $E_2^{5+qk, 2p^2+4p+1+k} = 0$ for $k \geq 1$. The argument is similar, using the same base-$p$ digit counting technique as above, so we omit it.

For the second part, it suffices to show that $E_2^{6+qk, 2p^2+4p+2+k} = 0$ for all $k \geq 0$. For $k = 0$ we have $4$ monomials $acAB$, $xcybA$, $acyzA$ and $xaybB$ in the $E_1$-term of the MSS. The first $3$ support a $\delta_1$-differential, and the kernel is spanned by $-xcybA + acyzA$ and $xaybB$. These are both $\delta_1$-boundaries. So the May $\delta_1$-differential kills all $4$ classes.

For $k \geq 1$ the argument is once again similar.
\end{proof}

\begin{prop} \label{p:longdiff}
Suppose $V(3)$ exists. In the ANSS converging to $\pi_* V(3)$ we have
\[
 d_{(p-2)q+1}(xacyBC^{p-2}) \doteq aybBA^K.
\]
\end{prop}

\begin{proof}
Consider the class $\theta = ayb B \in E_2^{5, 2p^2+4p+1}$. By Lemma \ref{l:smallstems} above it is an infinite cycle, and represents a non-zero element $\gamma \in \pi_{q(2p^2+4p+1) - 5} V(3)$. Let $K = p^2-2p+1$. Because $\beta_1$ is represented by $A$ in the ANSS and $\theta$ is an infinite cycle, $A^K \theta$ cannot support an outgoing differential and if $\beta_1^K \gamma = 0$ this class must be a boundary in the ANSS.

By Lemma \ref{l:smallstems} above, $\alpha_1 \gamma = 0$. Now we follow Ravenel \cite[Section 7.6]{Ra04}: Define $w_i = \langle \alpha_1,\ldots,\alpha_1, \beta_1^{ip} \rangle$ for $1 \leq i \leq p-2$, where we have $i+1$ factors of $\alpha_1$. Then by Toda-bracket juggling and $\langle \alpha_1,\ldots,\alpha_1 \rangle = \beta_1$ we have
\[
 \alpha_1 w_{p-2} = \langle \alpha_1,\ldots,\alpha_1 \rangle \beta_1^{(p-2)p} = \beta_1^{p^2-2p+1}.
\]
Ravenel concludes that $\beta_1^{p^2-p+1}$ is divisible by $\alpha_1 \beta_1^p$ and is therefore $0$. But the exact same argument shows that $\beta_1^{p^2-2p+1} \gamma$ is divisible by  $\alpha_1 \gamma$ and is therefore $0$.

This means that we have a differential $d_r(x) = aybBA^K$ for some $r$. We must have $r = kq+1$ for some $k \geq 1$. Our class is in bidegree $(2p^2-4p+7, p^3+5p+1)$, so the source is in bidegree $(2p^2-4p+6-kq, p^3+5p+1-k)$.

We are in the range where $E_2^{*,*} \cong H^*(P)$, so we compute the $E_2$-term using the MSS. For $k = 1$ the May $E_1$-page has a class $ybB^2A^{p^2-3p+1}$. It is a $\delta_1$-cycle, but it supports a $\delta_{2p-1}$-differential because of the $B^2$-factor so it does not survive to the ANSS.

For $2 \leq k \leq p-3$ the $E_1$-term of the MSS in this bidegree is $0$, so there are no potential differentials. For $k = p-2$ the situation is a little bit more complicated. On the $E_1$-term of the MSS we have $3$ classes $xacyBC^{p-2}$, $xaczABC^{p-3}$ and $xacbAC^{p-2}$. The last supports a $\delta_1$, while $\delta_{2p-1}(xacB^2C^{p-3}) \doteq xaczABC^{p-3} - xacyBC^{p-2}$. It follows that the $E_\infty$-term of the MSS in this bidegree is $1$-dimensional, so the $E_2$-term of the ANSS is $1$-dimensional and detected by $xacyBC^{p-2}$. For $k = p - 1$ we once again get nothing, and for $k \geq p$ the differential would originate from the $4$'th quadrant.

Hence the only possible way for $A^K \theta$ to die is the claimed differential.
\end{proof}

We can also argue topologically for this differential. In Ravenel's calculation, $w_{p-2}$ represents $\alpha_1 \beta_{p/p}^{p-2}$, and since $d_{q+1}(\beta_{p/p}) = \alpha_a \beta_1^p$ we get
\[
 d_{(p-1)q+1}(\beta_{p/p}(\alpha_1\beta_{p/p}^{p-1})) = \beta_1^{p^2-p+1}.
\]

In our situation the differential $d_{q+1}(\beta_{p/p}) = \alpha_1 \beta_1^p$ is replaced by the null-homotopy of $\alpha_1 \gamma$ represented, up to a unit, by $\xi = acyB$. (We can think of this as a $d_1$-differential in the ANSS, which makes sense if we use the cobar complex to define the $E_1$-term.) The class $\alpha_1$ is represented by $x$, and the class $\beta_{p/p}$ is represented by $C$, so the same argument as in Ravenel gives
\[
 d_{(p-2)q+1}(x \xi C^{p-2}) \doteq aybBA^K.
\]

\begin{thm} \label{t:nolongdiff}
Suppose $V(3)$ exists. In the ANSS converging to $\pi_* V(3)$ we have $d_{pq+1}(v_4) = 0$.
\end{thm}

\begin{proof}
The only alternative is $d_{pq+1}(v_4) \doteq aybBA^N$.

It follows from Proposition \ref{p:longdiff} that either $aybBA^N$ is already $0$ on the $((p-2)q+1)$-page of the ANSS or it is killed by a differential $d_{(p-2)q+1}(xacyBA^{p-3}C^{p-2}) \doteq aybBA^N$. In either case the class $aybBA^N$ disappears before the $(pq+1)$-page, so the result follows.

It is not logically necessary, but a similar argument as before shows that the only class in the ANSS that can hit $aybBA^N$ is $xacyBA^{p-3}C^{p-2}$, so we must have $d_{(p-2)q+1}(xacyBA^{p-3}C^{p-2}) \doteq aybBA^N$.
\end{proof}

This completes the proof of the first part of Theorem \ref{t:main}, since we have verified that $v_4$ is a permanent cycle and hence it represents a class $v_4 \in \pi_{2p^4-2} V(3)$. For the second part, we use that $V(3)$ is a ring spectrum to upgrade $v_4$ to a $V(3)$-self map as in Equation \ref{eq:upgradetoselfmap} and take the cofiber. This completes the proof.



\begin{thebibliography}{1}

\bibitem{Ma66}
J.~P. May.
\newblock The cohomology of restricted {L}ie algebras and of {H}opf algebras.
\newblock {\em J. Algebra}, 3:123--146, 1966.

\bibitem{OAI:V4}
{OpenAI}.
\newblock {A finite Smith--Toda complex $V(4)$ at the prime $1009$}.
\newblock OpenAI Math Release preprint
  \href{https://github.com/openai/math/blob/main/preprints/A-finite-Smith-Toda-complex-V4-at-the-prime-1009-September-23-2026/paper.pdf}{OAI:A-finite-Smith-Toda-complex-V4-at-the-prime-1009-September-23-2026},
  2026.

\bibitem{OAI:Vn}
{OpenAI}.
\newblock {Finite Smith--Toda Complexes at Varying Primes}.
\newblock OpenAI Math Release preprint
  \href{https://github.com/openai/math/blob/main/preprints/Finite-Smith-Toda-Complexes-at-Varying-Primes-September-23-2026/paper.pdf}{OAI:Finite-Smith-Toda-Complexes-at-Varying-Primes-September-23-2026},
  2026.

\bibitem{Ra04}
Douglas~C. Ravenel.
\newblock {\em Complex cobordism and stable homotopy groups of spheres}, volume
  347 of {\em AMS Chelsea Publishing}.
\newblock American Mathematical Society, Providence, RI, second edition, 2004.

\bibitem{To71}
Hirosi Toda.
\newblock On spectra realizing exterior parts of the {S}teenrod algebra.
\newblock {\em Topology}, 10:53--65, 1971.

\end{thebibliography}

\end{document}